\documentclass[12pt]{amsart}
\usepackage{enumerate, amssymb}
\usepackage[matrix,arrow,curve]{xy}
\usepackage{mathrsfs}
\usepackage{array}

\makeatletter
\@addtoreset{equation}{thm}
\makeatother

\input{diagrams.tex}
\diagramstyle[scriptlabels,height=8mm,width=8mm]
\def\pdfsyncstart{}
\def\pdfsyncstop{}
\theoremstyle{plain}

\newtheorem{thm}{Theorem}[section]
\newtheorem{cor}[thm]{Corollary}
\newtheorem{lem}[thm]{Lemma}
\newtheorem{prop}[thm]{Proposition}
\newtheorem{claim}[thm]{Claim}

\theoremstyle{definition}
\newtheorem{defi}[thm]{Definition}
\newtheorem{defis}[thm]{Definitions}
\newtheorem{conj}[thm]{Problem}
\newtheorem{conv}[thm]{Convention}
\newtheorem{nota}[thm]{Notation}
\newtheorem{rem}[thm]{Remark}
\newtheorem{rems}[thm]{Remarks}
\newtheorem{exa}[thm]{Example}
\newtheorem{exas}[thm]{Examples}
\newtheorem{sit}[thm]{}
\newtheorem{setup}[thm]{Setup}

\newcommand{\brem}{\begin{rem}}
\newcommand{\brems}{\begin{rems}}
\newcommand{\erem}{\end{rem}}
\newcommand{\erems}{\end{rems}}
\newcommand{\bexa}{\begin{exa}}
\newcommand{\bexas}{\begin{exas}}
\newcommand{\eexa}{\end{exa}}
\newcommand{\eexas}{\end{exas}}
\newcommand{\bdefi}{\begin{defi}}
\newcommand{\edefi}{\end{defi}}
\newcommand{\bdefis}{\begin{defis}}
\newcommand{\edefis}{\end{defis}}
\newcommand{\bcor}{\begin{cor}}
\newcommand{\ecor}{\end{cor}}
\newcommand{\blem}{\begin{lem}}
\newcommand{\elem}{\end{lem}}
\newcommand{\bconv}{\begin{conv}}
\newcommand{\econv}{\end{conv}}
\newcommand{\bconj}{\begin{conj}}
\newcommand{\econj}{\end{conj}}
\newcommand{\bprop}{\begin{prop}}
\newcommand{\eprop}{\end{prop}}
\newcommand{\bthm}{\begin{thm}}
\newcommand{\ethm}{\end{thm}}
\newcommand{\bnota}{\begin{nota}}
\newcommand{\enota}{\end{nota}}
\newcommand{\bsit}{\begin{sit}}
\newcommand{\esit}{\end{sit}}
\newcommand{\be}{\begin{equation}}
\newcommand{\ee}{\end{equation}}
\newcommand{\bproof}{\begin{proof}}
\newcommand{\eproof}{\end{proof}}
\def\ba{\begin{array}}
\def\ea{\end{array}}
\def\bea{\begin{eqnarray}}
\def\eea{\end{eqnarray}}

\def\bnum{\begin{enumerate}}
\def\enum{\end{enumerate}}
\renewcommand\labelenumi{\rm (\arabic{enumi})}

\newcommand{\Der}{\operatorname{Der}}

\newcommand{\Sing}{\operatorname{Sing}}

\newcommand{\supp}{\operatorname{supp}}

\newcommand{\divis}{\operatorname{div}}

\newcommand{\affcone}{\operatorname{AffCone}}
\newcommand{\Pic}{\operatorname{Pic}}
\newcommand{\Diff}{\operatorname{Diff}}

\newcommand{\Cl}{\operatorname{Cl}}

\newcommand{\rk}{\operatorname{rk}}

\newcommand{\A}{{\mathbb A}}
\newcommand{\PP}{{\mathbb P}}

\newcommand{\C}{{\mathbb C}}
\newcommand{\Q}{{\mathbb Q}}
\newcommand{\Z}{{\mathbb Z}}
\newcommand{\N}{{\mathbb N}}

\newcommand{\F}{{\mathbb F}}
\newcommand{\G}{{\mathbb G}}

\def\bdi{\pdfsyncstop\begin{diagram}}
\def\edi{\end{diagram}\pdfsyncstart}

\title[Affine cones over Fano threefolds]{Affine cones
over Fano threefolds and additive
group actions}

\author{Takashi Kishimoto}

\address{Takashi Kishimoto: Department of Mathematics,
Faculty of Science, Saitama University, Saitama 338-8570, Japan}
\email{tkishimo@rimath.saitama-u.ac.jp}

\author{Yuri Prokhorov}

\address{Yuri Prokhorov: Department
of Algebra, Faculty of Mathematics, Moscow State
University, Moscow 117234, Russia
\qquad  and\qquad
Laboratory of Algebraic Geometry, SU-HSE,
7 Vavilova Str., Moscow 117312, Russia
}
\email{prokhoro@gmail.com}

\author{Mikhail Zaidenberg}

\address{Mikhail Zaidenberg: Universit\'e
Grenoble I, Institut Fourier, UMR 5582 CNRS-UJF, BP 74, 38402 St.\
Martin d'H\`eres c\'edex, France} \email{zaidenbe@ujf-grenoble.fr}

\thanks{
The first author was supported by a Grant-in-Aid for Scientific
Research of JSPS No. 20740004. The second author  was partially supported by
RFBR, grant No. 11-01-00336-a, the grant of
Leading Scientific Schools, No. 4713.2010.1
and
AG Laboratory SU-HSE, RF government
grant, ag. 11.G34.31.0023. This work
was done during a stay of the second and the third authors at the Max
Planck Institut f\"ur Mathematik at Bonn and a stay of the first
and the second authors at the Institut Fourier, Grenoble. The
authors thank these institutions for hospitality.}

\date{}

\begin{document}

\begin{abstract}
We address the following question:

\smallskip

\noindent {\em When an affine cone over a smooth Fano threefold
 admits an effective action of the additive group?}

\smallskip

In this paper we deal with Fano threefolds of index 1 and Picard
number 1. Our approach is based on a geometric criterion from
\cite{KPZ}, which relates the existence of an additive group
action on the cone over a smooth projective variety  $X$ with the
existence of an open polar cylinder $U\simeq Z\times\A^1$ in $X$.
Non-trivial families of Fano threefolds carrying a cylinder were
found in \cite{KPZ}. Here we provide new such examples.

\end{abstract}
\maketitle

\bigskip

{\footnotesize \tableofcontents}

\section*{Introduction}
All varieties in this paper are defined over $\C$. It is known
\cite{KPZ} that the affine cone over any smooth del Pezzo surface
of degree $d\ge 4$ anticanonically embedded in $\PP^d$ admits an
effective $\G_a$-action. The existence of a  $\G_a$-action
on the affine cone over a projective variety $X$
depends upon the polarization chosen.
However, if
$\Pic(X)\simeq\Z$, then all polarizations are proportional and so all
the affine cones over $X$ simultaneously admit or do not admit
a  $\G_a$-action.

On the other hand, under the assumption $\Pic(X)\simeq\Z$ it is
natural to restrict to Fano varieties $X$ only, since otherwise
$X$ is not uniruled and so the affine cones over $X$ do not admit
a $\G_a$-action, see \cite{KPZ}. Consider, for instance, a Fano
variety $X$ with Picard number one which contains the affine space
$\A^n$ as a Zariski open subset. Clearly, every affine cone over
$X$ admits a $\G_a$-action. This applies e.g. to $\PP^n$, the
smooth quadric $Q$ in $\PP^{n+1}$, or the Fano threefold $X_5$ of
index 2 and degree 5. In \cite[5.1-5.2]{KPZ} we found two more
families of rational Fano threefolds $X$ with Picard number one
such that every affine cone over $X$ admits a $\G_a$-action.
Namely, these are the smooth intersections of two quadrics in
$\PP^5$ and the Fano threefolds $X_{22}$ of genus $12$. In the
next theorem we provide two more such families. Given a Fano
threefold $X$, we let $\tau(X)$ denote the Fano scheme of $X$ that
is, the component of the Hilbert scheme parameterizing the lines
on $X$.

\bthm\label{mthm} Let $X$ be a Fano threefold of genus $g=9$ or
$10$ with $$\operatorname{Pic}(X)=\mathbb Z\cdot (-K_X)\,.$$ If
the scheme $\tau(X)$ is not smooth, then the affine cone over $X$
under any projective embedding $X\hookrightarrow\PP^N$ admits an
effective $\G_a$-action. The Fano threefolds with a nonsmooth
scheme $\tau(X)$ form a codimension one subvariety in the
corresponding moduli space. \ethm

Let us make the following observation.
It is known  \cite{Pr-90c} that the automorphism group of a
Fano threefold $X$ as in Theorem \ref{mthm} is finite. It follows
that for any affine cone over $X$, the group of its linear
automorphisms is one-dimensional, while the whole automorphism
group is infinite-dimensional, see \cite[\S\S 2-3]{KPZ}.

A geometric construction used in the proof of Theorem
\ref{mthm} involves a line $L$ on $X$, which corresponds to a
non-smooth point of $\tau(X)$. Besides, in Theorems
\ref{theorem-degree4-DuVal} and \ref{theorem-degree3-DuVal} we
provide families of examples, which evoke instead a smooth point
$[L]\in\tau(X)$. It seems plausible that the latter families are
not contained in the former ones. A natural question arises
whether the conclusion of Theorem \ref{mthm} remains true for any
Fano threefold of genus $g=9$ or $10$ with Picard number $1$. We
expect, however, that the answer is negative.

The proof of Theorem \ref{mthm} is based on the following
geometric criterion. Let $X\subseteq\PP^n$ be a smooth projective
variety. We say that $X$ possesses a {\em polar $\A^1$-cylinder}
$U$ if there exists an effective $\Q$-divisor $D$ on $X$ such that
$D\sim_{\Q} H$, where $H$ stands for the hyperplane section, and
$$U=Y\setminus\supp\, D\cong Z\times\A^1\,$$ for some
quasiprojective variety $Z$. We let $\affcone(X)$ denote the
affine cone over $X$.

\bthm\label{crit}{\rm (\cite[Theorem 3.9]{KPZ})} If $X$ as above
possesses a polar $\A^1$-cylinder $U\to Z$ with $\Pic(Z)=0$, then
$\affcone(X)$ admits an effective  $\G_a$-action.

Vice versa, if $\affcone(X)$ admits an effective $\G_a$-action,
then there exists  in $X$ an open set $U=Y\setminus\supp\,D$,
where $D$ is as before, isomorphic to the total space of a line
bundle. \ethm

Specifying Theorem \ref{crit} we deduce the following corollary.

\bcor\label{crit3} Let $X$ be a smooth
subvariety in $\PP^n$ with
$\Pic(X)\simeq\Z$. Then $\affcone(X)$ admits an effective
$\G_a$-action if and only if there exists  in $X$ an open cylinder
$U\simeq Z\times \A^1$.\ecor

\bproof Indeed, since $\Pic(X)\simeq\Z$, every cylinder
in $X$ is polar. Since a line bundle over $Z$ is locally trivial,
shrinking $Z$ if necessary we may assume that it is
trivial.\eproof

We apply this criterion to smooth Fano threefolds of index one and
with Picard number one. Thus Theorem \ref{mthm} follows from
Theorem \ref{mthm1} which says that every Fano threefold $X$
satisfying the assumptions of Theorem \ref{mthm} has a cylinder.

Section 1 contains a brief overview on Fano threefolds, with a
special accent on the rationality problem. Besides, we collect
here some useful facts on the variety of lines in a Fano threefold. In
Section 2 we describe two standard constructions, which give all
Fano threefolds of genus $9$ and $10$. Sometimes the proofs are
hardly accessible in the literature, so we provide them here. The main
Theorems \ref{mthm}\footnote{See also Theorem \ref{mthm1}.},
\ref{theorem-degree4-DuVal}, and \ref{theorem-degree3-DuVal}
are  proven in Section 3.

\section{Generalities on Fano threefolds}\label{sect-gener}
We recall that a Fano variety is a smooth projective variety $X$
with an ample anticanonical class $-K_X$. The Fano index $r=i(X)$
is defined via $-K_X=rH$, where $H\in\Pic(X)$ is a  primitive
ample divisor class. It is well known that $r\le\dim X +1$. We
write $X=X_d$ for a Fano threefold of degree $d$, where  $d=H^3$.
The genus $g$ of $X$ is defined via $2g-2=-K_X^3\,(=dr^3)$.

\subsection{Classification of Fano threefolds: rationality.}
Any Fano threefold $X$ has index $r\le 4$. Furthermore,
\begin{enumerate}
\item[$\bullet$] if $r=4$ then $X\simeq \PP^3$;
\item[$\bullet$] if $r=3$ then $X\simeq Q$,
where $Q$ is a smooth quadric in $\PP^4$. \end{enumerate}

We assume in the sequel that $\Pic(X)\simeq\Z$.

\begin{enumerate}
\item[$\bullet$]
If $r=2$ then the degree of $X$ varies in the range
$d=1,\ldots,5$. More precisely,
\begin{enumerate}\item[(1)] if $d=1$ then $X$
is a hypersurface of degree $6$ in the weighted projective space
$\PP(1,1,1,2,3)$.  Such a threefold $X$ is non-rational
\cite{Tyu}, \cite{Gr};
\item[(2)] if $d=2$ then $X$ is a hypersurface of degree $4$ in
the weighted projective space $\PP(1,1,1,1,2)$. Such a threefold
$X$ is non-rational \cite{Vo};
\item[(3)] if $d=3$ then $X$
is a cubic hypersurface in $\PP^4$, which is known to be
non-rational \cite{CG};
\item[(4)] if $d=4$ then
$X=X_{2\cdot 2}$ is an intersection of two quadrics in $\PP^5$.
Such a threefold is rational  \cite{IPr};
\item[(5)]
if $d=5$ then $X=X_5$ is a linear section (by $\PP^6$) of the
Grassmanian $G(2,5)$ under its Pl\"ucker embedding in $\PP^9$.
Such a threefold is  rational and unique up to isomorphism
\cite{IPr}.
\end{enumerate}

\item[$\bullet$] If
$r=1$ then the genus of $X$ varies in the range
$g=2,\ldots,10$ and $12$. More precisely,
\begin{enumerate}
\item If $g=2,3,5$, or $8$,
then the threefold $X$ is non-rational (see \cite{Is5}, \cite{IPu}
for $g=2$, \cite{IM}, \cite{Is5} for $g=3$, \cite{Be} for $g=5$,
\cite{Is5} and \cite{CG} for $g=8$);
\item if $g=4$ or $6$ then a general threefold $X$ is non-rational \cite{Be},
\cite{IPu}, \cite{Tyu};
\item if $g=7,9,10$, or $12$
then $X$ is rational \cite{IPr}.
\end{enumerate}
\end{enumerate}

We are interested in Fano threefolds which possess a cylinder. By
the Castelnuovo rationality criterion for surfaces, such a
threefold must be rational. Of course, if $X$ contains the affine
space $\A^3$ as an open subset then it has a cylinder. Besides the
projective space $\PP^3$, a smooth quadric $Q$ in $\PP^4$, and the
Fano threefold $X_5$, also certain threefolds $X_{22}$ contain
$\A^3$ \cite{Fur}. The latter threefolds form a subvariety of
codimension two in the moduli space of all the $X_{22}$, which has
dimension 6. In contrast, a cylinder exists in every Fano
threefold $X_{22}$ or $X_{2 \cdot 2}$ \cite[\S 5]{KPZ}. In Theorem
\ref{mthm1} below we describe families of Fano threefolds with a
cylinder among the $X_{16}$ ($g=9$) and the $X_{18}$ ($g=10$).

The question arises whether every rational Fano threefold  carries
a cylinder; in particular, whether this is true for all the
threefolds $X_{12}$ ($g=7$), $X_{16}$ and $X_{18}$.

\smallskip

\subsection{Families of lines on Fano threefolds}
\label{sect-lines}
In the sequel we need the following facts.

\bthm[{\cite{Sh1}}, {\cite{Re1}}, {\cite[Ch.\ 3, \S 2]{Is4}},
{\cite[\S 4.2]{IPr}}]\label{theorem-lines}
Let $X=X_{2g-2}$ be a Fano threefold of genus $g\ge 3$ with
$\Pic(X)=\Z\cdot (-K_X)$,
 anticanonically embedded in $\PP^{g+1}$.
 Then the following hold.
 \begin{enumerate}
\item
There is a line $L$ on $X$.
 \item
For the normal bundle $\mathscr N_{L/X}$ there are the following
possibilities:
\[
\begin{array}{ll}
 (\alpha )&\mathscr N_{L/X}\simeq \mathscr O_{\mathbb P^1}
 \oplus\mathscr O_{\mathbb
 P^1}(-1), \quad\text{or}
\\[10pt]
 (\beta)&\mathscr N_{L/X}\simeq \mathscr O_{\mathbb P^1}(1)
 \oplus\mathscr O_{\mathbb
 P^1}(-2)\,.
\end{array}
\]
 \item
The scheme  $\tau(X)$ is of pure dimension   $1$.
 \item
The scheme   $\tau(X)$ is smooth  and reduced at a point $[L]\in
\tau(X)$ if and only if the corresponding line $L$ is of type
$(\alpha)$.

 \item
For $g\ge 7$ any line $L$ on $X$ meets at most a finite number of
other lines $L_i$ on $X$.
\end{enumerate}
\ethm

\begin{rem}\label{gen-red} Let $g=9$ or $10$ and $\Pic(X)=\Z\cdot (-K_X)$.
According to \cite{Pr-90b} and \cite{GLN} every irreducible
component of the scheme $\tau(X)$ is generically reduced. Thus for
a Fano threefold $X$ as in Theorem \ref{mthm}, the set of
non-smooth points of the scheme $\tau(X)$ is at most finite. On
the other hand, for a general Fano threefold $X$ of this type, the
scheme $\tau(X)$ is an irreducible smooth curve \cite[\S
3.2]{Pr}, \cite[Cor.\ 5.1.b]{Il}.
\end{rem}

\section{Fano threefolds of genera $9$ and $10$}
\label{sect-g=9-10}

We need the following lemma.

\blem\label{lemma-F}
\begin{enumerate}
 \item[(a)]
Any smooth curve $\Gamma$ of degree $7$ and genus $3$  in $\PP^3$
lies on a unique (irreducible) cubic surface $F=F(\Gamma)$ in
$\PP^3$.
 \item[(b)]
For any smooth, linearly non-degenerate curve $\Gamma$ of degree
$7$ and genus $2$  in $\PP^4$, the quadrics containing $\Gamma$
form a linear pencil, say, $\mathcal Q$. The base locus of this
pencil is an irreducible quartic surface $F=F(\Gamma)$ in $\PP^4$.
 \end{enumerate}
\elem

\begin{proof}
We provide a proof in the case $g=10$, the case $g=9$ being
similar. Let $\mathscr I_{\Gamma}$ be the ideal sheaf of
$\Gamma\subseteq \PP^4$. Using the exact sequence
\[
0\longrightarrow \mathscr I_{\Gamma}(2)\longrightarrow
\mathscr O_{\mathbb P^4}(2)
\longrightarrow \mathscr O_{\Gamma}(2)\longrightarrow 0
\]
 by Riemann-Roch we obtain that $\dim H^0(\mathscr I_{\Gamma}(2))\ge
 2$. Hence
there is a pencil of quadrics $\mathcal Q$ through $\Gamma$.

Assume to the contrary that there exist three linearly independent
quadrics $Q_1$, $Q_2$, and $Q_3\subseteq \PP^4$  passing through
$\Gamma$. Then $Q_1\cap Q_2\cap Q_3=\Gamma+L$ (as a scheme), where
$L$ is a line. Consider the exact sequence \be\label{ese}
0\longrightarrow\mathscr O_{\Gamma\cup L}\longrightarrow\mathscr
O_{\Gamma}\oplus\mathscr O_L \longrightarrow \mathscr F
\longrightarrow 0\,, \ee where the quotient sheaf $\mathscr F$ is
supported on $\Gamma\cap L$. Since
$$
\chi(\mathscr O_{\Gamma\cup L})=-4\quad \text{and}\quad
\chi(\mathscr O_\Gamma\oplus\mathscr O_L)=\chi(\mathscr
O_\Gamma)+\chi(\mathscr O_L)=0\,,
$$
we obtain by (\ref{ese})
\[
\# (\Gamma\cap L)=\dim H^0(\mathscr F)
= \chi(\mathscr O_\Gamma\oplus\mathscr O_L)- \chi(\mathscr
O_{\Gamma\cup L}) =4\,.
\]
Thus  $L$ must be a  $4$-secant line of $\Gamma$. Hence the
projection with center $L$ would map $\Gamma$ to a plane cubic, a
contradiction.

 Let us show finally that $F$ is irreducible. Indeed, otherwise
$\Gamma$ would be contained  in an irreducible surface $F'$ of
degree $\le 3$ in $\PP^4$. Since $\Gamma$ is assumed to be
linearly non-degenerate, $F'$ must be a cubic surface. By
\cite[Ch. 4, \S 3]{GH}, either $F'$ is a cone or $F'\simeq\F_1$.
Proceeding as at the beginning of the proof, it is easily seen
that in both cases $ h^0(\mathscr I_{\Gamma}(2))\ge
 h^0(\mathscr I_{F'}(2))\ge 3$. Hence
there is a two-dimensional family of quadrics passing through
$\Gamma$, which leads to a contradiction as before.
\end{proof}

In \ref{classification-F}--\ref{theorem-construction-g=9}
below we deal with the
following setting.

\begin{setup}\label{setup}
We consider the following two cases:
\begin{enumerate}
\renewcommand\labelenumi{(\roman{enumi})}
\renewcommand\theenumi{(\roman{enumi})}
 \item
For $g=9$, we let $W=\mathbb P^3$ and $\Gamma\subseteq \mathbb
P^3$ be a smooth non-hyperelliptic curve of degree $7$ and genus
$3$.
 \item
For $g=10$, we let $W=Q\subseteq \mathbb P^4$ be a smooth quadric
and $\Gamma$ be a smooth curve of degree $7$ and genus $2$ on $Q$.
\end{enumerate}
In both cases, we let $F=F(\Gamma)$ denote the corresponding
surface from Lemma \textup{\ref{lemma-F}}.
\end{setup}

In the next proposition we list the possibilities for such a
surface $F$.

\bprop\label{classification-F}
\renewcommand\labelenumi{\rm (\arabic{enumi})}
\renewcommand\theenumi{\rm (\arabic{enumi})}
In the notation and assumptions as in
\textup{\ref{lemma-F}}--\ref{setup} we let $g=9$ in
case (a) of Lemma \ref{lemma-F} and $g=10$ in case (b).
Then the surface
$F=F(\Gamma)\subseteq \PP^{g-6}$ belongs to one of the following
classes.
\begin{enumerate}
 \item \label{classification-F-normal}
$F\subseteq  \PP^{g-6}$ is a normal del Pezzo surface with at
worst Du Val singularities; or
 \item \label{classification-F-nonnormal}
$F\subseteq  \PP^{g-6}$ is a non-normal scroll, whose singular
locus $\Lambda=\operatorname{Sing}(F)$ is a double line.
Furthermore, the normalization $F'$ of $F$ is a smooth scroll $F'$
of the minimal degree $g-6$ in $\PP^{g-5}$, and the normalization
map $\nu: F'\to F$ is induced by the projection from a point $P\in
\PP^{g-5}\setminus F'$. The restriction $\nu|_{\nu^{-1}(\Lambda)}:
\nu^{-1}(\Lambda)\to \Lambda$ is a ramified double cover. There
are the following possibilities.
\begin{enumerate}
\renewcommand\labelenumii{\rm (\alph{enumii})}
\renewcommand\theenumii{\rm (\alph{enumii})}
 \item
If $g=9$ then $F'\simeq \mathbb F_1$, the embedding $F'\subseteq
\PP^{4}$ is defined by the linear system $|\Sigma+2\ell|$ on
$\mathbb F_1$, where $\Sigma\subseteq\mathbb F_1$ is the
exceptional section and $\ell$ is a ruling, and
$\nu^{-1}(\Lambda)\sim \Sigma+\ell$ is a reduced conic on
$F'\subseteq \PP^{4}$, which is either smooth or degenerate.\\ If
$g=10$ then one of the following hold.
 \item
$F'\simeq \mathbb F_0=\PP^1\times \mathbb P^1$, the embedding
$F'\subseteq \PP^{5}$ is defined by the linear system
$|\Sigma+2\ell|$, and $\nu^{-1}(\Lambda)\sim \Sigma$ is a smooth
conic on $F' \subseteq \PP^{5}$; or
 \item[{\rm (b${}^{\prime}$)}]
$F'\simeq \mathbb F_2$, the embedding $F'\subseteq \PP^{5}$ is
defined by the linear system $|\Sigma+3\ell|$, and
$\nu^{-1}(\Lambda)\sim \Sigma+\ell$ is a reduced degenerate conic
on $F' \subseteq \PP^{5}$.
\end{enumerate}
\end{enumerate}
\eprop

\begin{proof}
Since $F$ is a complete intersection, it is Gorenstein. By the
adjunction formula $\omega_F\simeq \mathscr O_F(-1)$, i.e. $F$ is
(possibly non-normal) del Pezzo surface.

If $F$ is normal, then by \cite{HW}  $F$ is either a surface
described in \ref{classification-F-normal}, or a cone over an
elliptic curve $C\subseteq \PP^{g-7}$ of degree $g-6$. Assume to
the contrary that $F$ is a cone. Let $\xi: \tilde F\to F$ be the
blowup of the vertex. Then $\tilde F$ is a smooth ruled surface
over $C$. Let as before $\Sigma$ and $\ell$ be the exceptional
section and a ruling, respectively, with $\Sigma^2=-k$. Letting
$M=\xi^*\mathscr O_{F}(1)$ and $\tilde \Gamma$ be the proper
transform of $\Gamma$ on $\tilde F$, we can write $M\equiv \Sigma+
k \ell$ and $\tilde \Gamma\equiv a \Sigma +b \ell$. Then
\[
\begin{array}{l}
0=M\cdot \Sigma, \qquad g-6=M^2= k, \qquad  \Sigma^2=-k=6-g\,,
\\[10pt]
7=\tilde \Gamma\cdot M= b, \quad\text{and}\quad \tilde \Gamma\cdot
\Sigma=a(6-g)+7\ge 0\,.
\end{array}
\]
Since $\tilde \Gamma\simeq\Gamma$ is not an elliptic curve, $a\ge
2$. This is only possible for $g=9$,  $a=2$, and so $k=3$. On the
other hand, by adjunction
\[
2g( \tilde \Gamma)-2= (\tilde \Gamma +K_{\tilde F})\cdot \tilde \Gamma=
8,
\]
a contradiction, since $g(\tilde \Gamma)=g(\Gamma)\le 3$.

If $F$ is non-normal then by \cite[Theorem 8]{Na}, \cite{Re3},
\cite[9.2.1]{Dol}, $F$ is a projection of a normal surface $F'$ of
the minimal degree $g-6$ in $\PP^{g-5}$. It is well known (see
e.g., \cite[Ch.\ 4, \S 3, p.\ 525]{GH}) that $F'\subseteq
\PP^{g-5}$ is either a Veronese surface $F'_4\subseteq \PP^5$, or
the image of a Hirzebruch surface $\mathbb F_n$ under the map
given by the linear system $|\Sigma+k\ell|$, where $2k-n=g-6$ and
$k\ge n$. The case of the Veronese surface is impossible because
the degree of every curve on $F'_4\subseteq \PP^5$ is even. Thus
$F'\simeq \mathbb F_n$. Let $\Gamma'\subseteq \mathbb F_n$ be the
proper transform of $\Gamma$ on $F'$. We can write $\Gamma'\sim
a\Sigma+b\ell$, where $a\ge 2$ and $b\ge na$. Note that in the
case $g=9$ we have $a\ge 3$, since $\Gamma$ is assumed being
non-hyperelliptic, see \ref{setup}. It is easily seen that the
remaining possibilities are as in
\ref{classification-F-nonnormal}.
\end{proof}

The following corollary is immediate.

\begin{cor}\label{newcor}
In the notation of Proposition
\textup{\ref{classification-F}}\textup{\ref{classification-F-nonnormal}},
the class of $\Gamma'$ in the Picard group of the normalization
$F'\simeq \mathbb
F_n$ is as follows:
\begin{enumerate}
\renewcommand\labelenumi{\rm (\alph{enumi})}
\renewcommand\theenumi{\rm (\alph{enumi})}
 \item
$g=9$, $F'\simeq \mathbb F_1$, $\Gamma'\sim 3\Sigma+4\ell$;
 \item
$g=10$, $F'\simeq \mathbb F_0$,
$\Gamma'\sim 2\Sigma+3\ell$;
 \item[{\rm (b${}^{\prime}$)}]
$g=10$, $F'\simeq \mathbb F_2$,
$\bar \Gamma'\sim 2\Sigma+5\ell$.
\end{enumerate}
In all cases $\Lambda$ is a $(13-g)$-secant line of $\Gamma$ i.e.,
a $3$-secant if $g=10$ and $4$-secant if $g=9$.
\end{cor}

Now we can strengthen part (b) of Lemma \textup{\ref{lemma-F}}.

\blem\label{rem} In case (b) of Lemma \textup{\ref{lemma-F}} the
pencil $\mathcal Q$ contains a smooth quadric.\elem

\begin{proof}
Assume to the contrary that every quadric $Q  \in    \mathcal Q$ is
singular. By Bertini Theorem a general member $Q  \in \mathcal
Q$ is smooth outside $F$. Since $F$ is a complete intersection,
every member  $Q  \in    \mathcal Q$ is smooth  at the points of
$F\setminus \operatorname{Sing}(F)$. If $F$ has at worst isolated
singularities, then so does every quadric $Q  \in \mathcal
Q$. Moreover, in this case they all must have a common singularity.
Hence $F$ should be a cone,
which contradicts Proposition \ref{classification-F}.

Thus under our assumption $F$ must have non-isolated
singularities. Moreover, by Proposition \ref{classification-F}(2)
$F$ must be singular along a line $\Lambda$. If some quadric $Q
\in \mathcal Q$ is singular along $\Lambda$, then $F$ is again a
cone, which is impossible. Thus we may assume that every quadric
$Q \in \mathcal Q$ has an isolated singular point $P\in \Lambda$.
Fixing such a quadric $Q$, we can choose an affine chart  in
$\PP^4$ with coordinates $x_1,\dots,x_4$ centered at $P$ so that
$\Lambda$ is given by $x_1=x_2=0$ and $Q$ is given by
$x_1x_3+x_2x_4=0$. There is a quadric $Q'\in  \mathcal Q$ given by
$x_1 u(x_1, x_2,x_4)+x_2v(x_1,x_2, x_3,x_4)=0$, where $u$ and $v$
are linear forms. Since $F$ is singular along $\Lambda$,  at every
point of $\Lambda$ the Jacobian matrix of these two quadratic
forms has rank $\le 1$. Therefore $x_3v(0,0, x_3,x_4)=x_4u(0,
0,x_4)$ for all $x_3$, $x_4$. This implies that $v(0,0,
x_3,x_4)=u(0, 0,x_4)=0$. So $Q'$  is given by $x_1 (ax_1+bx_2)
+x_2(cx_1+dx_2)=0$ for some $a, b, c,d\in\C$. Hence $Q'$ is a cone
with vertex $\Lambda$. Therefore $F=Q\cap Q'$ is a cone with
vertex $P=(0,0,0,0)\in \Lambda$, which again gives a contradiction
and concludes the proof.
\end{proof}

In the case of a curve $\Gamma$
lying on a smooth surface $F$, the following
result can be found in \cite{Is3}.
In the present more general form, the result
was announced without proof
in \cite[Theorems 4.3.3 and 4.3.7]{IPr}.
Besides, we can quote an explanation in \cite[4.3.9(ii)]{IPr}
as to why the assumption in \ref{setup}(i) that
the curve $\Gamma$ is non-hyperelliptic is important.
The details of the proof can be found
in an unpublished thesis \cite{Pr} (in Russian).
For the reader's convenience,
we reproduce them below;
see also the (unpublished) notes \cite{BL}.

\bthm
\label{theorem-construction-g=9}
In the notation as in Setup \textup{\ref{setup}}
there exists a Sarkisov link
\be\label{main-diag}
\xymatrix{
&\tilde D\ar[dl]\ar@{^{(}->}[rr]&&\tilde
X\ar[dr]^{\sigma_0}\ar[dl]_{\sigma}\ar@
{-->}[rr]^{\chi}&
& \hat X\ar[dl]_{\varphi_0}\ar[dr]^{\varphi}&&\hat
F\ar[dr]\ar@{_{(}->}[ll]
\\
\Gamma\ar@{^{(}->}[rr]&&W\ar@{-->}@/_19pt/ [rrrr]_{\psi}&&X_0&&X&&
L\ar@{_{(}->}[ll] } \ee where $\sigma$ is the blowup of $\Gamma$,
$\sigma_0$ and $\varphi_0$ are the anticanonical maps onto
$X_0\subseteq \PP^{g-1}$, $\chi$ is a flop, $X=X_{2g-2}$
 is a smooth
Fano threefold of genus $g$ with $\Pic(X)=\Z\cdot (-K_X)$
anticanonically embedded in $\PP^{g+1}$, and $\varphi$ is the
blowup of a line $L$ on $X$. The exceptional divisor $\hat F$ of
$\varphi$ is a proper transform of the surface
$F=F(\Gamma)\subseteq W$. The exceptional divisor $\tilde D$ of
$\sigma$ is a proper transform of a divisor $D\in
|-(12-g)K_X-(25-2g)L|$. The map $\psi^{-1}$ is the double
projection with center $L$ that is, a map given by the linear
system $|A-2L|$ on $X$, where $A\sim -K_X$ is a hyperplane
section. \ethm

\bproof
Let $\sigma: \tilde X\to W$ be the blowup of $\Gamma$. Let $\tilde
D$ be the exceptional divisor and let $H^*=\sigma^*H$, where $H$
is the positive generator of $\Pic (W)\simeq \Z$.
We have (see e.g. \cite[Lemma 2.2.14]{IPr})
\be\label{numdata}
(H^*)^3=g-8, \qquad
(H^*)^2\cdot \tilde D=0,\qquad
H^*\cdot \tilde D^2=-H\cdot \Gamma=-7\,,
\ee and
\[
\tilde D^3=-\deg \mathscr N_{\Gamma/W}=
\begin{cases}
-23&\text{if $g=10$,}
\\
-32&\text{if $g=9$.}
\end{cases}
\]
Letting $\tilde F\subseteq \tilde X$ be the proper transform of
$F$ we get $\tilde F\sim (12-g)H^*-\tilde D$. The divisor classes
$-K_{\tilde X} \sim (13-g)H^*-\tilde D$ and $\tilde F$ form a
basis of $\operatorname{Pic}(\tilde X)\simeq \mathbb Z\oplus
\mathbb Z$. We have
\be\label{num-data}
-K_{\tilde X}^3=2g-6>0,\quad (-K_{\tilde X})^2\cdot \tilde F=3,
\quad -K_{\tilde X}\cdot \tilde F^2=-2, \quad\text{and}\quad
\tilde F^3=g-13\,.
\ee
We need the following fact.

\begin{claim}\label{claim-nef-big}
The divisor $-K_{\tilde X}$ is nef and big.
\end{claim}
 \begin{proof}
Since $-K_{\tilde X}^3=2g-6>0$, the divisor $-K_{\tilde X}$ is
big. To establish that it is also nef, we
consider the case $g=10$; the proof in the case $g=9$ is similar.
From the exact sequence
\[
0\longrightarrow \mathscr I_{\Gamma}(3)\longrightarrow
\mathscr O_{W}(3)
\longrightarrow \mathscr O_{\Gamma}(3)\longrightarrow 0\,
\]
we obtain by Riemann-Roch
\[
\dim H^0(\mathscr I_{\Gamma}(3)) \ge
\dim H^0(\mathscr O_{W}(3))-
\dim H^0(\mathscr O_{\Gamma}(3))=10\,.
\] The members of the linear system $|-K_{\tilde X}|$
are proper transforms of the members of the linear system
$|-K_W|=|\mathscr O_W(3)|$ passing through $\Gamma$. Hence
\be\label{9} \dim |-K_{\tilde X}|\ge 9\,.\ee Applying Lemma
\ref{lemma-F} it is easily seen that the only reducible members
$\tilde G\in |-K_{\tilde X}|$ are those of the form $\tilde
G=\tilde F+H^*$. Hence such divisors form a linear subsystem in
$|-K_{\tilde X}|$ of codimension $\ge 5$.

Assume to the contrary that there exists an irreducible
curve $\tilde C$ on $\tilde X$ with
$\tilde C\cdot ({-}K_{\tilde X})<0$, and let
$C=\sigma(\tilde C)\subseteq W$.
Since $g(\Gamma)=2$,
the curve $\Gamma$ does not admit any $4$-secant line.
Indeed, otherwise the projection
from this line would send $\Gamma$ isomorphically to a plane cubic,
which is impossible. Since $$\# (C\cap \Gamma) =
\tilde C\cdot \tilde D> 3H^*\cdot \tilde C =3\deg C\ge 3\,,$$
the curve $C$ cannot be a line.
If $C$ is contained in a plane $\Pi\subseteq \PP^4$
then by the same argument
\[
\# (\Pi \cap \Gamma) \ge \# (C\cap \Gamma) > 3\deg C\ge 6\,.
\]
Since $\deg \Gamma =7$ and $\Gamma$ is linearly non-degenerate,
we get a contradiction.
Thus $C$ is not contained in a plane and so $\deg C\ge 3$.
Assume that $C$ is contained in some hyperplane $\Theta\subseteq\PP^4$.
Then as above
\[
\# (\Theta \cap \Gamma) \ge \# (C\cap \Gamma) > 3\deg C\ge 9\,,
\]
which again leads to a contradiction because $\deg \Gamma=7$.
Therefore $C$ is linearly non-degenerate and $\deg C\ge 4$.

On the other hand, $F$ contains a line, say, $\Upsilon$.
Let $\tilde \Upsilon\subseteq \tilde X$
be its proper transform.
We have $\tilde \Upsilon\cdot ({-}K_{\tilde X})\le 3=
\Upsilon\cdot ({-}K_W)$. Therefore,
fixing four general points on $\tilde \Upsilon$,
a member $\tilde M\in |{-}K_{\tilde X}|$ passing through these points
is forced to contain $\tilde \Upsilon$.
The family of all such members has codimension at most 4, while
degenerate ones vary in a family of codimension at least five,
as we observed before. Hence there exists an
irreducible divisor  $\tilde M\in |{-}K_{\tilde X}|$
containing $\tilde \Upsilon$.
By our assumption  $\tilde M \cdot \tilde C<0$, and then also
$\tilde F\cdot \tilde C=\tilde M \cdot \tilde C-H^*
\cdot \tilde C<0$.
Thus the intersection  $\tilde M\cap \tilde F$
contains  $\tilde C\cup \tilde \Upsilon$ and so by (\ref{numdata})
\[
\deg  (C+\Upsilon)=(\tilde C+\tilde \Upsilon)\cdot H^* \le
\tilde M\cdot\tilde F\cdot H^*=-K_{\tilde X}\cdot\tilde F\cdot H^*
=(3H^*-\tilde D)\cdot (2H^*-\tilde D)\cdot H^*=5\,.
\]
It follows that $\deg C=4$, so $C\subseteq \PP^4$ is a rational
normal quartic curve.
Every quadric in the linear system $H^0(\mathscr I_{C\cup \Gamma}(2))$
contains $C\cup \Gamma $. Picking two distinct points on $\Gamma$
let us consider the family of quadrics from $H^0(\mathscr I_{C}(2))$
passing through these points. It has dimension four.
Such a quadric cuts $\Gamma$ in $13+2=15$ points, hence contains it.

An easy computation gives $\dim H^0(\mathscr I_C(2))=6$.
It follows that
$$\dim H^0(\mathscr I_{C\cup \Gamma}(2))\ge 6-2=4\,.$$
However, the latter
contradicts Lemma \ref{lemma-F}(b). This shows that in the case $g=10$,
the divisor
$-K_{\tilde X}$ is nef. The case $g=9$ can be treated similarly.
\end{proof}

By the Base Point Freeness Theorem we deduce the following.

\begin{cor}\label{BPF}
For some $n>0$ the linear system $|-nK_{\tilde X}|$ defines
a birational morphism $\sigma_0: \tilde X\to X_0\subseteq \PP^N$
whose image is a Fano threefold with at worst
Gorenstein canonical singularities.
Moreover $-K_{\tilde X}=\sigma_0^* (-K_{X_0})$.
\end{cor}

Our next claim is as follows.

\begin{claim}\label{claim-2.9}
The morphism $\sigma_0$ is small, i.e. it does
not contract any divisor.
\end{claim}

\bproof Assume that $\sigma_0$ contracts a prime divisor $\Xi \sim
\alpha({-}K_{\tilde X})- \beta \tilde F$. Then by (\ref{num-data})
\[
0= \Xi\cdot ({-}K_{\tilde X})^2= (2g-6)\alpha - 3\beta\,.
\]
This yields $\beta= (2g/3-2)\alpha$. Since $\Xi\neq\tilde F$ and
 $-K_{\tilde X}$ is nef by \ref{claim-nef-big}, we have
\[
0\le  \Xi\cdot\tilde F\cdot ({-}K_{\tilde X})=3\alpha+2\beta=\alpha (4g/3-1)\,.
\]
Hence $\alpha>0$. Furthermore,
\[
\Xi \sim  \alpha( 2g^2/3 - 11g + 37  ) H^*   + \alpha(2g/3-3)\tilde D\,.
\] Since $\sigma_*\Xi$ is effective we must have $2g^2/3 - 11g +
37\ge 0$, a contradiction. \eproof

The following corollary is standard.

\begin{cor}\label{next-cor}
In the notation as above, $X_0$ has at worst isolated compound
Du Val singularities.
\end{cor}

Following the techniques outlined in \cite[\S 4.1]{IPr} we can now
finish the proof of Theorem \ref{theorem-construction-g=9}.

\medskip

\noindent {\em End of the proof of
\ref{theorem-construction-g=9}.} If $-K_{\tilde X}$ is ample then
the map $\sigma_0$ is an isomorphism. In this case we let $\hat
X=\tilde X=X_0$ and $\chi$ to be the identity map. Otherwise by
\cite{Kol3} the contraction $\sigma_0: \tilde X\to X_0$ can be
completed to a flop triangle as in diagram (\ref{main-diag}). Here
$\varphi_0$ is another small resolution of $X_0$. Let $\hat
C\subseteq \hat X$ and $\tilde C\subseteq \tilde X$ be the flopped and the
flopping curves, respectively. Then $\chi$ induces an isomorphism
$\tilde X\setminus \tilde C\simeq \hat X\setminus \hat C$.

In both cases the
divisor $-K_{\hat X}=\varphi^*(-K_{X_0})$ is nef and big.
Further, we have
\[
-K_{\hat X}^3=-K_{\tilde X}^3=2g-6,\quad
(-K_{\hat X})^2\cdot \hat F=(-K_{\tilde X})^2\cdot \tilde F=3, \quad
-K_{\hat X}\cdot \hat F^2=-K_{\tilde X}\cdot \tilde F^2=-2.
\]
Since $\operatorname{\Pic}(\hat X)\simeq
\operatorname{\Pic}(\tilde X)$ is of rank $2$ the Mori cone
$\operatorname{NE}(\hat X)$ is generated by two extremal rays. One
of them has the form $\mathbb R_{+}[T]$, where $T$ is a curve in
the fiber of $\sigma$ (resp., $\varphi_0$) if $\chi$ is an
isomorphism (resp., not an isomorphism). Let $R\subseteq
\operatorname{NE}(\hat X)$ be the second extremal ray. Since
$-K_{\hat X}$ is nef and big, $R$ is $K$-negative. By \cite{Mo}
there exists a contraction $\varphi : \hat X\to X$ of $R$.

Since $-K_{\tilde X}-\tilde F=\sigma^*\,\mathscr{O}(1)$ is nef we
have $(-K_{\tilde X}-\tilde F) \cdot \tilde C>0$. Therefore
$\tilde F\cdot \tilde C<0$ and $\hat F\cdot \hat C>0$. Since
$-K_{\hat X}\cdot \hat F^2=-2<0$, the divisor $\hat F$ is not nef.
Hence $\hat F\cdot R<0$ that is, the ray $R$ is not nef. By the
classification of extremal rays \cite{Mo}, $\varphi$ is a
birational divisorial contraction. Moreover, the
$\varphi$-exceptional divisor coincides with $\hat F$. If
$\varphi\colon \hat X\to X$ contracts $\hat F$ to a point, then by
\cite{Mo}
\[
({-}K_{\hat X})^2\cdot \hat F= 4,\quad 2\quad\text{or}\quad 1\,.
\]
On the other hand, $({-}K_{\hat X})^2\cdot \hat F=3$, a contradiction.
Hence $\varphi\colon \hat X\to X$ contracts $\hat F$
to a curve $Z$. In this case both $X$ and $Z$ are smooth and $\varphi$
is the blowup of $Z$
\cite{Mo}.
Moreover,  $X$ is a Fano threefold of Fano index $r=1$, $2$, $3$ or $4$.
The group $\Pic \hat X$ is generated by
$\hat F$ and
\[
-\frac1r\varphi^*K_X=\frac1r({-}K_{\hat X}+\hat F)\,.
\]
Therefore, the subgroup generated by $\tilde F$ and ${-}K_{\tilde
X}$ has index $r$ in $\Pic \tilde X\simeq \Pic \hat X$. This
implies that $r=1$. We have
\begin{multline*}
({-}K_X)^3=({-}K_{\hat X})\cdot({-}K_{\hat X}+\hat F)^2=
\\
=({-}K_{\hat X})^3+2\hat F\cdot ({-}K_{\hat X})^2+({-}K_{\hat X})
\cdot F^{\prime 2}=
\\
=2g-6+6-2=2g-2\,,
\end{multline*}
i.e. $X$ is a Fano threefold of genus  $g$. Furthermore,
\[
\deg Z={-}K_X\cdot Z=({-}K_{\hat X}+\hat F)\cdot \hat F\cdot
({-}K_{\hat X})=3-2=1\,,
\]
i.e. $Z\subseteq X$ is a line. Now an easy computation shows that
$\hat F^3\neq \tilde F^3$, so $\chi$ is not an isomorphism.

By \cite[Prop.\ 3]{Is4} the linear system $|-K_{\hat X}|$ defines
a birational map and $X_0$ is a Fano threefold with at worst
isolated Gorenstein terminal singularities. In particular,
$|-K_{X_0}|$ is very ample. Hence the linear system $|-K_{\tilde
X}|=\sigma_0^* |-K_{X_0}|$ is base point free   and defines the
map $\sigma_0$.

Finally, $\Gamma$ is (as a scheme) the base locus of the linear
subsystem $\sigma_*|-K_{\tilde X}|\subseteq |\mathscr O_W(13-g)|$.
It remains to show that in the case $g=9$ the curve $\Gamma$ is
not hyperelliptic. Assume the converse. It was shown already that
$\Gamma$ does not admit a $5$-secant line. On the other hand, by
\cite[Ch.\ 2, \S 5]{GH} $\Gamma$ admits a $4$-secant line, say,
$N$. The projection from $N$ defines a linear system of degree $3$
and dimension $\ge 1$ on $\Gamma$. Hence the curve $\Gamma$ is
hyperelliptic and trigonal. However, this is impossible, since
otherwise the linear systems $g_2^1$ and $g_3^1$ on $\Gamma$ define
a birational morphism $\Gamma \to \PP^1\times \PP^1$
whose image is a divisor of bidegree $(2,3)$.
This contradicts the assumption that $g(\Gamma)=3$.
Now the proof of Theorem \ref{theorem-construction-g=9}
is completed.
\end{proof}

\bcor\label{isom} In the  notation as above we have $X\setminus
D\simeq W\setminus F$. \ecor

In the next proposition we describe the flopped and  the flopping
curves in (\ref{main-diag}).

\bprop\label{proposition-4-secants}
In the  notation as above we let $\tilde C\subseteq \tilde X$
and $\hat C\subseteq \hat X$ be the flopping and the
flopped curve, respectively.
Then the following hold.
\begin{enumerate}
\item
Any irreducible component $\hat C_i\subseteq \hat X$ either is
a proper transform of a line $L_i\neq L$ on $X$ meeting $L$,
or (in the case where $L$ is of type $(\beta)$) is
the negative section $\Sigma$ of the ruled surface $\hat F\simeq
\mathbb F_3$.
\item
The curve $\hat C$ is a disjoint union of the $\hat C_i$'s.
\item
For any $\hat C_i$ we have
$$
\mathscr N_{\hat C_i/\hat X}\simeq \mathscr
O_{\PP^1}(-1) \oplus\mathscr O_{\PP^1}(-1)\quad\text{
or}\quad \mathscr N_{\hat C_i/\hat X}\simeq \mathscr O_{\PP^1}\oplus\mathscr
O_{\PP^1}(-2)\,.
$$
It follows that $\chi$ coincides with the Reid's pagoda
\textup{\cite{Re2}} near each $\hat C_i$.
\item
The curve $\tilde C$ in $\tilde X$ is a disjoint union of the
$\tilde C_i$'s, where each $\tilde C_i$ is the proper transform of
a $(13-g)$-secant line of $\Gamma$.
\end{enumerate}
\eprop

\begin{proof}
Recall that $\hat C$ and $\tilde C$ are exceptional loci of
$\varphi_0$ and $\sigma_0$, respectively \cite{Kol3}. The
assertion (1) is proven in \cite[Proposition 3, (iv)]{Is3} and
\cite[Proposition 4.3.1]{IPr}, while (2) and (3) in
\cite[Proposition 4]{Cu} and \cite[Corollary 12, Theorem 13]{Cu},
respectively. By virtue of (3) $\tilde C$ is a disjoint union of
its irreducible components. Finally $(-K_{\hat X}-\hat F)\cdot
\hat C_i=-1$. Therefore $1=(-K_{\tilde X}-\tilde F)\cdot \tilde
C_i=\sigma^*\mathscr O_{W}(1)\cdot\tilde C_i$. So $\sigma(\tilde
C_i)$ is a line. Since $-K_{\tilde X}\cdot \tilde C_i=0$, this
line must be $(13-g)$-secant.
\end{proof}

Finally in the next theorem we provide a criterion as to when the
surface $F$ in Theorem \ref{theorem-construction-g=9} is normal.

\bthm\label{normal} In the notation of Theorems \ref{theorem-lines} and
\textup{\ref{theorem-construction-g=9}}, the surface $F$ is normal
 if and only if $L$ is a line of type $(\alpha)$ on $X$. \ethm

\bproof
We use the notation of Proposition \ref{proposition-4-secants}.
Assume that  $L$ is of type $(\beta)$, and let $\tilde C_0$ denote
the proper transform on $\tilde X$ of the negative section $\Sigma$
of the ruled
surface $ \hat F\simeq \mathbb F_3$. By Remark 5.13 in \cite{Re2},
$\tilde F$ is not normal
along $\tilde C_0$. Since $\tilde C_0$ is
a smooth rational curve, $\sigma$ is an isomorphism at a general
point of $\tilde C_0$. So $F$ is also non-normal along
$\sigma(\tilde C_0)$.

Assume to the contrary that  $L$ is of type $(\alpha)$, while $F$
is non-normal. Then $F$ is singular along a line $\Lambda$.
Clearly $\Lambda\neq \Gamma$, so $\tilde F$ is also non-normal and
singular along $\sigma^{-1}(\Lambda)$. The map $\chi$ is an
isomorphism near a general ruling $\hat f\subseteq \hat F\simeq
\mathbb F_1$. Letting $\tilde f= \chi^{-1}(\hat f)$, the surface
$\tilde F$ is smooth along $\tilde f$ and $\sigma_0(\tilde
f)=\varphi_0(\hat f)$ is a line on $\sigma_0(\tilde
F)=\varphi_0(\hat F)\simeq \mathbb F_1$. Let $l\subseteq F$ be a
general line on  a non-normal scroll $F$
 and  $\tilde l$  be its proper transform on $\tilde F$. An easy
computation shows that $\sigma_0(\tilde l)$ is again a line on
$\sigma_0(\tilde F)=\varphi_0(\hat F)\simeq \mathbb F_1$. Thus we
may suppose that $\tilde l=\tilde f$. On the other hand, $\tilde
l\cap \operatorname{Sing}(\tilde F)\neq \emptyset$, a
contradiction. \eproof

\section{Constructions of cylinders}\label{sect-constr}

In this section we prove Theorem \ref{mthm}. Recall that under its
assumptions $X=X_{2g-2}$ is a Fano threefold in $\PP^{g+1}$ of
genus $g=9$ or $10$ with $\operatorname{Pic}(X)=\mathbb Z\cdot
(-K_X)$, having a non-smooth Fano scheme $\tau(X)$. By virtue of
Corollary \ref{crit3} the first assertion of  Theorem \ref{mthm}
is equivalent to the following one.

\bthm\label{mthm1} Under the assumptions of Theorem \ref{mthm} the
variety $X$ contains a cylinder.\ethm

\bproof Assuming that the scheme $\tau(X)$ is not smooth at a
point $[L]\in \tau(X)$, it suffices to construct a cylinder in
$W\setminus F$ (see Corollary \ref{isom}).

By Theorem \ref{theorem-lines}(4) $L$ is a line of type $(\beta)$
on $X$. According to Theorem \ref{normal} the surface $F$ is
non-normal, and so by Proposition \ref{classification-F}
$\Lambda=\Sing(F)$ is a double line on $F$. Consider the following
diagram:
\[
\xymatrix{
&\bar W\ar[dr]^{q}\ar[dl]_{p}&
\\
W\ar@{-->}[rr]^{\xi}&&\PP^{g-8} }
\]
where $\xi$ is the projection from $\Lambda$, $p$ is the blowup of
$\Lambda$, and $q=\xi \circ p$. We show below that $q$ is a
$\PP^{11-g}$-bundle over $\PP^{g-8}$. Let $\bar E\subseteq \bar W$
be the exceptional divisor and $\bar F\subseteq \bar W$ be the
proper transform of $F$.

In the case $g=10$ the fibers of $\xi$ are intersections of our
smooth quadric $W\subseteq \PP^4$ (see \ref{setup}) with planes in
$\PP^4$ containing $\Lambda$. Therefore $q$ is a $\PP^1$-bundle
over $\PP^2$, whose fibers are the proper transforms of lines in
$W\subseteq \PP^4$ meeting $\Lambda$. The morphism $q : \bar W\to
\mathbb P^2$ is given by the linear system $|p^*\mathscr O_W(1)-
\bar E|$. Since $\bar F\sim 2p^*\mathscr O_W(1)- 2\bar E$, the
image $q(\bar F)=\xi (F)$ is a conic on $\mathbb P^2$. Since
$\mathscr N_{\Lambda/W}\simeq \mathscr O_\Lambda \oplus \mathscr
O_\Lambda(1)$, the $\PP^1$-bundle $\bar E\to\Lambda$ is that of
the Hirzebruch surface $\F_1\to\PP^1$. Moreover, its negative
section $\bar\Sigma$ is a fiber of $q$. It follows that the open
set $W\setminus F\simeq \bar W\setminus (\bar F \cup \bar E)$ is
an $\A^1$-bundle over $\PP^2\setminus q(\bar F\cup \bar\Sigma)$.
By \cite[Theorem 2]{KM} and \cite[Theorem]{KW}, this bundle is
trivial over a Zariski open subset $Z\subseteq\PP^2\setminus
q(\bar F\cup \bar\Sigma)$. This gives a cylinder contained in
$W\setminus F$ and also a cylinder on $X$.

In the case $g=9$ the fibers of $\xi$ are planes in $W= \PP^3$.
The intersection of such a plane with the cubic surface $F$
consists of the double line $\Lambda$ and a residual line $l$.
Therefore $q$ is a $\PP^2$-bundle over $\PP^1$, and $\bar F\cup
\bar E$ intersects each fiber along a pair of lines.

More precisely, we have $\bar E \cong \F_0$ and $\bar F \cong
\F_1$ (see Proposition \ref{classification-F}(2a)). Furthermore,
$q|_{\bar E}$ and $q|_{\bar F}$, respectively, yield $\mathbb
P^1$-bundles with rulings being lines in the fibers of $q$. By a
simple computation we obtain that $\bar F|_{\bar E}\sim 2
\bar\Sigma + \bar l$, where $\bar\Sigma$ (resp. $\bar l$) is a
section (a ruling, respectively) of the trivial $\PP^1$-bundle
$\bar E \to \Lambda$. Notice that $\bar \Sigma$ is a line in a
fiber of $q$ and $\bar l$ is a section of $q$. The finite map
$p|_{\bar F}:\bar F \to F$ yields a normalization of $F$. For the
curve $\bar F |_{\bar E}$ there are the following two
possibilities :
\begin{enumerate}

\item[{\rm (i)}]
$\bar F|_{\bar E}=\Delta_1$, where $\Delta_1 \in |2\bar\Sigma
+\bar l|$ is irreducible, or

\item[{\rm (ii)}]
$\bar F|_{\bar E}=\bar\Sigma +\Delta_0$, where $\Delta_0 \in
|\bar\Sigma +\bar l|$ is a diagonal.

\end{enumerate}
We claim that $W\setminus F\simeq \bar W\setminus (\bar F \cup
\bar E)$ contains a cylinder. In what follows we deal with case
(ii) only; (i) can be treated in a similar fashion. There exists
exactly one fiber of $q$, say $\bar\Pi_\infty$, such that $\bar E,
\bar F$ and $\bar \Pi_\infty$ meet along a common line. Blowing up
$\bar W^\circ := \bar W\backslash \bar \Pi_\infty$ along the
irreducible curve $\bar E \cap \bar F \cap \bar W^{\circ}$, we
obtain an $\F_1$-bundle $\hat\pi:\hat W^{\circ} \to \A^1$ together
with the proper transforms $\hat F^{\circ}$ and $\hat E^{\circ}$
on $\hat W^{\circ}$ of $\bar F$ and $\bar E$, respectively. The
exceptional divisor $\hat E_1^{\circ}$ is ruled over $\A^1$ with
rulings being the $(-1)$-curves in the fibers isomorphic to
$\F_1$. There is a natural $\PP^1$-bundle structure $\rho:\hat
W^{\circ} \to\hat E_1^{\circ}$ which defines in each fiber of
$\rho$ the ruling $\F_1\to\PP^1$. The map $\rho$ sends $\hat
E^{\circ}$ and $\hat F^{\circ}$ to the intersections $\hat
E^{\circ} \cap\hat E_1^{\circ}$ and $\hat F^{\circ}\cap\hat
E_1^{\circ}$, respectively. The complement $\hat W^{\circ}
\setminus (\hat E_1^{\circ} \cup\hat E^{\circ} \cup\hat F^{\circ})
\simeq \bar W \backslash (\bar E \cup \bar F \cup \bar \Pi_\infty)
\simeq W\setminus (F\cup \Pi_\infty)$ is again a $\PP^1$-bundle
over $\hat E_1^{\circ}\setminus(\hat E^{\circ} \cup\hat
F^{\circ})$, where $\Pi_\infty:=p_*(\bar \Pi_\infty)$. This bundle
is trivial over a Zariski open subset $Z\subseteq\hat
E_1^{\circ}$, and admits a tautological section defined by $\hat
E_1^{\circ} \hookrightarrow\hat W^{\circ}$. After trivialization
the map $\rho:\rho^{-1}(Z)\to Z$ becomes the first projection
$Z\times\PP^1\to Z$. The second projection of the tautological
section defines a morphism $f:Z\to\PP^1$. The automorphism
$t\longmapsto (t-f(z))^{-1}$ of $Z\times\PP^1$ sends this section
to the constant section `at infinity'. The $\A^1$-bundle
$\rho:\hat W^{\circ}\setminus \hat E_1^{\circ}\to\hat E_1^{\circ}$
being trivial over $Z$ it defines a cylinder
$\rho^{-1}(Z)\setminus \hat E_1^{\circ}\simeq Z\times\A^1$, as
required. \eproof

\begin{proof}[Proof of Theorem \textup{\ref{mthm}}]
The first assertion of Theorem \ref{mthm} is a consequence of
Theorems \ref{theorem-lines}(4), \ref{normal}, and \ref{mthm1}.
Let us show the second one. Recall that the automorphism group of
a Fano threefold of genus $g=9$ or $10$ with
$\operatorname{Pic}(X)= (-K_X)\cdot \mathbb Z$ is finite
\cite{Pr-90c}.

Fix a moduli space $\mathscr M_g$ of the Fano threefolds of genus
$g=9$ or $10$ with $\operatorname{Pic}(X)= (-K_X)\cdot \mathbb Z$.
It can be defined using GIT, and is unique up to a birational
equivalence. Let $\mathscr{ML}_g$ be the moduli space of pairs
$(X,L)$, where $X$ is a Fano threefold as above and $L$ is a line
on $X$. Consider a natural projection $\pi: \mathscr{ML}_g\to
\mathscr M_g$ whose fiber over a point $[X]\in \mathscr M_g$
(which corresponds to $X$) is isomorphic to $\tau(X)$. By Theorem
\ref{theorem-lines}(3) we have $\dim \mathscr M_g=\dim
\mathscr{ML}_g-1$. By Theorem \ref{theorem-construction-g=9}
$\mathscr{ML}_g$ is isomorphic to the moduli space of embedded
curves $\Gamma \subseteq W$ of degree $7$ and genus
$g(\Gamma)=12-g$.

Let further $\mathscr{M}_g'\subseteq \mathscr{M}_g$ be the closed
subvariety formed by all Fano threefolds $X$ whose Fano scheme
$\tau(X)$ is non-smooth, and let $\mathscr{ML}_g'\subseteq
\mathscr{ML}_g$ be the subvariety formed by all pairs $(X,L)$ such
that $L$ is of type $(\beta)$. Then
$\mathscr{M}_g'=\pi(\mathscr{ML}_g')$. Since such a Fano threefold
$X$ contains at most a finite number of $(\beta)$-lines (see
Remark \ref{gen-red}) we have $\dim \mathscr{M}_g'=\dim
\mathscr{ML}_g'$. Now the second assertion of Theorem \ref{mthm}
is immediate in view of the following claim.
\end{proof}

\begin{claim}\label{12}
Let $\mathscr H_g$ be the Hilbert scheme parameterizing the curves
$\Gamma$ on $W$ of degree $7$ and arithmetic genus
$p_a(\Gamma)=12-g$. Then $\dim \mathscr H_g= 91-7g$. If the
surface $F=F(\Gamma)$ is smooth along $\Gamma$, then $\mathscr
H_g$ is smooth at the corresponding point. Furthermore, the
subscheme of $\mathscr H_g$ parameterizing the curves $\Gamma$
with $F(\Gamma)$ non-normal, has codimension $2$.
\end{claim}
\begin{proof}
Assuming that $F(\Gamma)$ is smooth  along $\Gamma$, we consider
an exact sequence of normal bundles with base
$\Gamma$\be\label{es} 0\longrightarrow\mathscr
N_{\Gamma/F}\longrightarrow\mathscr N_{\Gamma/W}
\longrightarrow\mathscr N_{F/W}|_\Gamma\longrightarrow 0\,. \ee
Taking into account the relations
\begin{equation*}
\deg \mathscr N_{\Gamma/F}=2g(\Gamma)-2+\deg \Gamma
\quad\text{and}\quad \deg \mathscr N_{F/W}|_\Gamma=\Gamma\cdot F\,,
\end{equation*}
we obtain by (\ref{es}) that $H^1(\mathscr N_{\Gamma/W})=0$ and
$\dim H^0(\mathscr N_{\Gamma/W})=91-7g$.
Now the first two assertions follow by the standard facts
of the deformation theory.

The proof of the last assertion is just a parameter count. By
Corollary \ref{newcor} the dimension of the family of curves
$\Gamma$ with a non-normal surface $F(\Gamma)$ equals $13$ and
$11$ in cases (a) and (b)-(b${}'$), respectively, while the family
of all non-normal surfaces  $F$ is of codimension $15-g$.
\end{proof}

\textbf{Second construction.} In this and the next subsections we
describe some families of Fano threefolds of genera $9$ and $10$
carrying a cylinder, which plausibly are not covered by Theorem
\ref{mthm}. In this subsection we prove the following theorem.

\bthm\label{theorem-degree4-DuVal}
In the  notation as in Setup \textup{\ref{setup}} and Theorem
\ref{theorem-construction-g=9}, in the case $g=10$
the threefold $X$ contains a cylinder whenever the surface $F$ has a
singularity worse than the Du Val singularity of type $A_1$.
\ethm

\bproof Assume that the surface $F\subseteq W\subseteq \PP^4$ is
singular, where $W$ is as before a smooth quadric in $\PP^4$ and
$F$ is a complete intersection quartic surface in $W$. Let $P\in
F$ be a singular point. There is a commutative diagram
 \be\label{codi}
\xymatrix{ &\bar W\ar[dr]^{q}\ar[dl]_{p}&
\\
W\ar@{-->}[rr]^{\xi}&&\PP^{3} } \ee where $\xi$ is the projection
from $P$ and  $p$ is the blowup of $P$.
Let $\bar E\subseteq \bar W$ be the exceptional divisor and $\bar
F\subseteq \bar W$ the proper transform of $F$. Then $\Pi=q(\bar
E)$ is a plane in $\PP^3$, while the birational morphism $q$ is
the blowup of a conic $C\subseteq \Pi$. Furthermore, let
$H_P=W\cap T_{P,W}$ be the tangent hyperplane section and $\bar
H_P\subseteq \bar W$ be its proper transform. Then   $\bar H_P$ is
the $q$-exceptional divisor. Now let $\bar F\subseteq \bar W$ be
the proper transform of $F$ and let $F^{\circ}=q(\bar F)$. It is
easily seen that $F^{\circ}\subseteq \mathbb P^3$ is a quadric.
Obviously, $W\setminus (F\cup H_P)\simeq \PP^3\setminus
(F^{\circ}\cup \Pi)$. Note that $\bar F\cap \bar E$ is the
exceptional divisor of $p_{\bar F}: \bar F\to F$ and $\bar F\cap
\bar E\simeq F^{\circ}\cap \Pi$.

Now assume that the singularity $P\in F$ is worse than a Du Val
singularity of type $A_1$. Then $\bar F\cap \bar E\simeq
F^{\circ}\cap \Pi$ cannot be a smooth conic. So it is either a
pair of crossing lines or a double line. In any case
$\PP^3\setminus (F^{\circ}\cup \Pi)$ admits a cylinder by the
arguments in the proof of Theorem \ref{mthm1} for $g=9$. Indeed,
$F^{\circ}\cup \Pi$ can be regarded as a cubic surface singular
along a line. \eproof

Consider, for instance, the following construction.

\begin{exa}
Let $\Gamma_{0}\subseteq \PP^2$ be a plane quartic curve with a
node $P_1$. Pick a pair of distinct general points $P_2$, $P_3\in
\Gamma_0$. Let $F_1\to \PP^2$ be the blowup of $P_1, P_2 , P_3$
and let $E_i$ be the corresponding exceptional divisors. Let
$\Gamma_1\subseteq F_1$ be the proper transform of $\Gamma_0$, and
let $P_4=\Gamma_1\cap E_2$ (this is a single point). Let $F_2\to
F_1$ be the blowup of $P_4$, $E_4$ be the exceptional divisor, and
$\Gamma_2\subseteq F_2$ be the proper transform of $\Gamma_1$.
Take a general point $P_5\in E_4$. Letting $F_3\to F_2$ be the
blowup of $P_5$ and $\Gamma_3\subseteq F_3$ be the proper
transform of $\Gamma_2$, we see that $F_3$ is a weak del Pezzo
surface of degree $4$ \cite[ch. 8]{Dol} containing two
$(-2)$-curves $C_2$ and $C_4$ that meet at a point. These are the
proper transforms of $E_2$ and $E_4$, respectively. The
anticanonical image of $F_3$ is a del Pezzo surface $F\subseteq
\PP^4$ with a Du Val singularity of type $A_2$, which is the image
of $C_2\cup C_4$. Since $\Gamma_3 \cdot (C_2+C_4)=1$, the image
$\Gamma$ of $\Gamma_3$ is a smooth curve of genus $2$ and degree
$7$. Thus $(F,\Gamma)$ satisfies the conditions of Theorem
\ref{theorem-degree4-DuVal}. More precisely, the complement $W
\backslash F$ contains a cylinder, and the center $\Gamma\subseteq
F$ of the blow-up $\sigma :\tilde{X} \to W$ is such that one can
reach a pair $(X,D)$ consisting of a Fano threefold $X=X_{18}$
$(g=10)$ and an irreducible divisor $D$ on $X$, which is the
proper transform of $\sigma^{-1}(\Gamma)$ on $X$, with
$X\backslash D \simeq W \backslash F$.
\end{exa}

\begin{rem}
The construction (\ref{codi}) works as well in the case of a
non-normal $F$. We believe that in this case there are several
cylinder structures on $X$, and hence the Makar-Limanov invariant
of any affine cone over $X$ vanishes.
\end{rem}

\textbf{Third construction.} In this subsection we construct
a cylinder in the complement
of an irreducible cubic surface $F\subseteq  \PP^3$
under certain restrictions on the singularities of
$F$.
In \cite{Oh} some families of cubic surfaces $F$ in $\PP^3$
were found such that the complement
$\PP^3\setminus F$ contains an $\A^2$-cylinder.
However, sometimes an $\A^1$-cylinder exists
while an $\A^2$-cylinder does not.

\bthm\label{theorem-degree3-DuVal} In the  notation as in
\textup{\ref{setup}}--\ref{theorem-construction-g=9}, in the case
$g=9$ the threefold $X$ contains a cylinder whenever the cubic
surface $F\subseteq\PP^3$ has a singular point of type $A_3$ or
worse. There exists a family of Fano threefolds $X$ satisfying
these assumptions. \ethm

This theorem follows from the next proposition and Example
\ref{c-gamma} below.

\bprop\label{cyl-cubic} Let $F$ be an irreducible
cubic surface in $\PP^3$. Then the complement
$\PP^3\setminus F$ contains a cylinder whenever the surface $F$ has a
singularity worse than the Du Val $A_2$ singularity.
\eprop

Before dwelling in the proof, let us mention an application of
this result.

\brem\label{ga-action} We observe that, whenever the complement of
a cubic surface in $\PP^3$ contains a cylinder, this complement
admits an effective $\G_a$-action. This applies e.g. to the
singular cubic surfaces as in Proposition \ref{cyl-cubic} or in
Lemma \ref{linear-cylinder} below.

More generally, let $X$ be a normal affine variety such that the
class group $\Cl(X)$ is a torsion group, and let $U\simeq
\A^1\times Z$ be an $\A^1$-cylinder in $X$. We claim that $X$
admits an effective $\G_a$-action along the corresponding
$\A^1$-fibration. Indeed, consider the $\G_a$-action on $U$ by
shifts on the second factor, and let $\partial\in\Der (\mathcal
O(U))$ be the corresponding locally nilpotent derivation. By our
assumption, a multiple of the effective reduced divisor
$D=X\setminus U$ is principal i.e., $mD=\divis(f)$ for some $f\in
\mathcal O(X)$ and $m\in\N$. Clearly, $f\in\ker (\partial)$ since
$f$ does not vanish on the $\A^1$-rulings of $U$. Hence
$f^N\delta$ is well defined and locally nilpotent on $\mathcal
O(X)$ for $N$ sufficiently large (cf.\ \cite[Proposition
3.5]{KPZ}). \qed\erem

We start the proof of Proposition \ref{cyl-cubic}
with several remarks and lemmas.

\brems\label{sing-cub} (1) Any  non-normal, irreducible cubic
surface $F$ in $\PP^3$ different from a cone is a scroll in lines
with a double line \cite{Na}, \cite{Re3} (cf.\ Proposition
\ref{classification-F}). The proof of Proposition \ref{cyl-cubic}
goes for such a surface $F$ in the same way as that of Theorem
\ref{mthm1}.

(2) If $F$ has a singular point $P$ of multiplicity $\ge 3$,
then $F$ is a cone over a plane cubic curve. So the projection
$\PP^3\setminus\{P\}\to \PP^2$ with center $P$
determines a (linear) cylinder
structure over an appropriate open set $Z\subseteq\PP^2$.

(3) In case (1) $F$ does not admit any isolated singularity. In
fact, if $F$ has a Du Val singularity then all its singular points
are at most isolated Du Val singularities. The classification of
all possible sets of Du Val singularities on cubic surfaces in
$\PP^3$ is as follows (see e.g., \cite{BW} or
\cite{Dol})\footnote{ The coefficients in the list mean the number
of singular points of a given type.}:
$$(nA_1), \,\,n=1,\ldots,4,\,\,(nA_2), \,\,n=1,2,3,\,\,
(A_3),\,\,(A_4),\,\,(A_5),\,$$
$$(nA_1,A_2),\,\,(nA_1,A_3),\,\,n=1,2,\,\,(A_1,2A_2),
\,\,(A_1,A_4),\,\,(A_1,A_5),\,\,$$
$$(D_4),\,\,(D_5),\,\,(E_6)\,.$$
\erems

In the proof of Proposition \ref{cyl-cubic} we use the following
simple observation.

\blem\label{linear-cylinder} Let $F$ be a cubic surface in
$\PP^3$, $L$ a line on $F$, and $\Pi_{\lambda}$
($\lambda\in\PP^1$) the pencil of planes through $L$. Suppose that
for a general $\lambda\in\PP^1$ \be\label{star}\Pi_{\lambda}\cap
F=L+C_{\lambda}, \quad\text{where}\quad C_{\lambda}\cap L=2P,\,\ee
i.e. $C_{\lambda}$ is a plane conic tangent to $L$ at a point $P$.
Then $\PP^3\setminus F$ contains a cylinder. \elem

\bproof Blowing up $\PP^3$ with center $L$
yields a diagram
\[
\xymatrix{ &\tilde {\mathbb P}^3\ar[dl]_{p}\ar[dr]^{q}&
\\
\mathbb P^3\ar@{-->}[rr]^{\xi}&&\mathbb P^1 }
\]
where $p$ is the blowup of $L$, $\xi$ is the projection with
center $L$, and $q$ is a $\PP^2$-bundle. Let $\tilde F$ be the
proper transform  of $F$ on $\tilde{\PP}^3$ and $\tilde E\subseteq
\tilde{\PP}^3$ be the exceptional divisor of $p$. We fix a member,
say, $\Pi_\infty$ of our pencil, and we let $\tilde\Pi_\infty$
denote its proper transform on $\tilde{\PP}^3$. In $\tilde{\PP}^3$
we consider the open set
$$\tilde U=\tilde{\PP}^3\setminus
(\tilde\Pi_\infty\cup\tilde
E)\simeq\PP^3\setminus\Pi_\infty\simeq\A^3\,.$$ Let $h$ be a
regular function on $\tilde U$ which defines the affine surface
$\tilde F\cap \tilde U$. Consider further a rational map
$$\delta=(q,h):\tilde{\PP}^3\dashrightarrow
\PP^1\times\PP^1\,.$$ Its restriction to the open set $\tilde
U\setminus\tilde F$ is regular, while the restriction to a general
fiber $\tilde\Pi_\lambda\setminus (\tilde E\cup \tilde F)$ of
$q|\tilde U$ defines an $\A^1$-fibration. Hence $\delta$ defines
as well an $\A^1$-fibration over a Zariski open subset of
$\PP^1\times\PP^1$. By\cite[Theorem 2]{KM} and \cite[Theorem]{KW}
there exists a cylinder in $\PP^3\setminus F$ compatible with this
$\A^1$-fibration.

\eproof

\brems\label{lines-cubics} (1) The construction in the proof
yields a cylinder in conics with a unique base point $P$. Such a
cylinder can exist only if $P\in F$ is a singular point.

(2) If $F$ as in Lemma \ref{linear-cylinder}
is singular at $P$, then
there is a line $L$ on $F$ through $P$.
Indeed, in an affine chart
centered at $P$ the equation of $F$ can be
written as $f_2+f_3=0$,
where $f_2$ and $f_3$ are homogeneous forms
of degree $2$ and $3$, respectively.
The system of equations $f_2=f_3=0$
defines $6$ lines on $F$ through $P$,
counting with multiplicities.

(3) Suppose that for a triple $(F,L,P)$ as before, the pencil
$\Pi_{\lambda}$ does not satisfy the assumptions of Lemma
\ref{linear-cylinder}. Then in an appropriate affine chart with
coordinates $(x,y,z)$ centered at a singular point $P$ of $F$, the
surface $F$ can be given by equation
$$xy+zg(x,y,z)=0\,.$$ Since the quadratic part is of rank $\ge 2$,
in this case $(F,P)$
is an $A_n$-singularity. These observations lead to the following corollary.
\erems

\bcor\label{notAn} If $(F,P)$ as before
is a Du Val singularity not of type $A_n$,
then $\PP^3\setminus F$ contains an $\A^1$-cylinder in conics with
a unique base point $P$.
\ecor

It remains to determine for which $A_n$-singularities $(F,P)$
of cubic surfaces the complement $\PP^3\setminus F$
contains a cylinder.

\blem\label{summer} Let $F$ be a cubic surface in $\PP^3$ with an
$A_n$-singularity $P\in F$. If $n\ge 3$ then the complement
$\PP^3\setminus F$ contains a cylinder.\elem

\bproof Suppose that $n\ge 3$, and let $f=f_2+f_3=0$ be an
equation of $F$ in a local  affine chart $(x,y,z)$ centered at
$P$. If $\rk f_2=1$ then $(F,P)$ is of type $D_n$ or $E_6$. If
$\rk f_2=2$ then $(F,P)$ is non-normal or of type $A_n$ ($n\ge
2$), and if $\rk f_2=3$ then $(F,P)$ is of type $A_1$. In the
former case, the result follows from Corollary \ref{notAn}. The
case $n\le 2$ is eliminated by our assumption. In the second case,
we can reduce the equation to the form
$$f=xy+g_3(x,y)+g_2(x,y)z+g_1(x,y)z^2+cz^3=0\,,$$
where $g_i$ is a homogeneous form of degree $i$. We claim that
$c=0$. Indeed, let the blowup of $\PP^3$ at $P$ be given in an
affine chart as $(x,y,z)\longmapsto (xz,yz,z)$, with the
exceptional divisor $z=0$. Then the equation of the proper
transform $F'$ of $F$ in this chart is
$$xy+g_3(x,y)z+g_2(x,y)z+g_1(x,y)z+cz=0\,.$$ Since $n>2$
and the surface $F'$ should acquire a singular point
of type $A_{n-2}$ at the origin, we conclude that $c=0$.

Furthermore, we may assume that $L=\{x=y=0\}$.
Consider the pencil of planes $\Pi_\lambda=\{y=\lambda x\}$
through $L$. We have $\Pi_\lambda\cap F=L+C_\lambda$,
where $L\cap C_\lambda=\{x=0, \,z^2=0\}$ has a double point.
Now the conclusion follows by Lemma \ref{linear-cylinder}. \eproof

\brem\label{first-rem} In the case where $P\in F$
is an $A_1$ or $A_2$ singularity and $L$ is a line on $F$
through $P$, there is no plane $\Pi_\lambda$ through $L$
such that the residual conic on the section $\Pi_\lambda\cap F$
is tangent to $L$ at $P$. \erem

\brem\label{almost-last-rem}
For a cubic surface $F\subseteq \PP^3$
the following are equivalent:
\begin{enumerate}
 \item
$F$ has a singularity worse than Du Val $A_2$ singularity,
 \item
there exists a line $L $ on $F$ such that the pair $(F,L)$ is not
purely log terminal (PLT).
\end{enumerate}
Indeed, assuming that all  singularities of $F$ are of type $A_1$
or $A_2$, consider a line $L$ on $F$. Since $L$ is smooth, for any
singular point $P\in F$ the dual graph of the minimal resolution
has the form
\[
\circ\text{---}\overset{L}{\bullet} \qquad
\text{or}
\qquad
\circ\text{---}\circ\text{---}\overset{L}\bullet
\]
Thus $(F,L)$ is PLT by the classification
of the  PLT singularities of surfaces.
Hence (2) implies (1).

To show the converse, assume that $(F,L)$ is PLT. Again by the
classification of the  PLT singularities, and because there is a
line passing through any singular point of $F$, the surface $F$ is
normal and has only $A_n$-singularities. Take $L$ as in Lemma
\ref{linear-cylinder}, and let $P\in L$ be a singular point of
$F$. For a general plane $\Pi$ passing through $L$ we have
$\Pi\cap F=L+C$, where $C$ is a smooth conic tangent to $L$ at
$P$. Then the pair $(F,C+L)$ is not log canonical (LC) at $P$.

On the other hand, we claim that the dual graph of the minimal
resolution of $(F,C+L)$ has the form
\[
\overset{L}\bullet \text{---}\overset{E_1}\circ\text{---}\cdots
\text{---}\overset{E_n}\circ\text{---}\overset{C}\bullet
\]
Consequently, the pair $(F,C+L)$ is LC at $P$, a contradiction.

To show the claim we consider the minimal resolution $\mu: \tilde
F\to F$  of the $A_n$-singularity $(F,P)$ and its fundamental
cycle $Z=\sum_{i=1}^n E_i$. Since  $L$ and $C$ both are smooth and
pass through $P$ we have $L\cdot Z=1=C\cdot Z$. Hence  $C$ and $L$
are both attached at the end vertices of the dual resolution chain
$\sum_{i=1}^n E_i$. It remains to show that they are attached at
the opposite end vertices. Write $\mu^*(C+L)=C'+L'+\sum_{i=1}^n
\alpha_i E_i$, where $\alpha_i>0$, $i=1,\ldots,n$, are the
vanishing orders on the $E_i$ of the pullback to $\tilde F$ of the
local equation of the Cartier divisor $C+L$ on $F$. Taking
intersections with $E_i$,  $i=1,\ldots,n$,  yields a system
\[-2\alpha_1+\alpha_2=-\delta_1,\,
\alpha_1-2\alpha_2+\alpha_3=-\delta_2,\,
\alpha_2-2\alpha_3+\alpha_4=-\delta_3,\, \ldots,\,
\alpha_{n-1}-2\alpha_n=-\delta_n\,,\] where $\delta_i=(C'+L')\cdot
E_i\in\{0,1,2\}$. We have $\sum_{i=1}^n \delta_i=2$. Assuming that
$\delta_1>0$ and summing up the equations we obtain
$-(\alpha_1+\alpha_n)=-2$, hence $\alpha_1=\alpha_n=1$. Plugging
in this in our system we find $\alpha_2+\delta_1=2$, where
$\alpha_2>0$ and $\delta_1>0$, hence $\alpha_2=1=\delta_1$. From
the second equation we deduce
$$\alpha_3=2\alpha_2-\alpha_1-\delta_2=1-\delta_2>0\,,$$
hence $\delta_2=0$ and $\alpha_3=1$, and so on. By recursion,
finally we get
$$\delta_1=1,\,\delta_2=\ldots=\delta_{n-1}=0,\quad\text{and}\quad
\delta_n=1\,.$$
Now the claim follows.  \erem

The  next example fixes the second part of
Theorem \ref{theorem-degree3-DuVal}.

\bexa\label{c-gamma} Let us construct a pair $(F,\Gamma)$, where
$\Gamma$ is a smooth curve of degree $7$
 and genus $3$ in the smooth locus of a cubic surface $F$ in $\PP^3$ with
a unique singular point $\Sing (F)=\{P\}$, such that $(F,P)$  is
an $A_3$-singularity.

 Consider a smooth quartic curve $\bar\Gamma$ in $\PP^2$.
 Blowing up a point $P_0$ on $\bar\Gamma$ and three infinitesimally
 near points $P_1,P_2,P_3$ on the subsequent proper transforms of
 $\bar\Gamma$,
 and then also a point $P_4\neq P_0$ on $\bar\Gamma$
 and an extra point $P_5\in\PP^2\setminus \bar\Gamma$,
 we obtain a smooth surface $\tilde F$,
 a chain of rational curves $\mathcal L=E_0+E_1+E_2$  on $F$
 with $E_i^2=-2$,
 $i=0,1,2$, which consists of the first three  components appeared
 in the exceptional locus,
 and a smooth curve $\tilde \Gamma$  on $\tilde F$ of genus $3$
 and anticanonical degree $7$, disjoint with $\mathcal L$.
 Blowing down $\mathcal L$ leads to a singular cubic surface $F$
 with a unique
 singular point of type $A_3$ anticanonically embedded in $\PP^3$.
 The image $\Gamma$ of $\tilde\Gamma$ on $F$ is a desired curve.
 \eexa

The following observation shows however that not any cubic surface
with a deep singularity is available for our purposes.

\brem\label{restr-sing} By construction, the criterion of Theorem
\ref{theorem-degree3-DuVal} on the existence of a cylinder in
$X=X_{16}\subseteq\PP^{10}$ ($g=9$) is valid as long as the cubic
surface $F$ in $\PP^3$ as in
\ref{setup}--\ref{theorem-construction-g=9} contains a smooth
curve $\Gamma$ of genus $3$ and degree $7$. However, there is no
such curve $\Gamma$ on a cubic surface $F$ with an isolated conic
singularity or a Du Val $E_6$ singularity (although by Proposition
\ref{cyl-cubic} in this case $\PP^3\setminus F$ contains an
$\A^1$-cylinder). In other words, a normal cubic surface with a
conic or an $E_6$ singularity cannot appear via the Sarkisov link
as in \ref{theorem-construction-g=9}. In the conic case, this
follows from Proposition \ref{classification-F}.

Suppose further that $F \subseteq\PP^{3}$ is a cubic surface with
a Du Val $E_6$ singularity $P$. Let $L$ be a line on $F$ passing
through $P$. Then $\Cl(F)\simeq\Z[L]$, where $L^2=1/3$. Since
$\deg (\Gamma)=7$ we have $\Gamma\sim 7L$ and so $\Gamma^2=49/3$.
It follows that $P\in\Gamma$.

Consider the minimal resolution $\sigma:\tilde F\to F$,
and let $$Z=E_1+2E_2+3E_3+2E_4+E_5+2E_6\,$$ be the fundamental cycle
supported on the exceptional divisor $E=\sum_{i=1}^6 E_i$ of
$\sigma$ with the dual graph
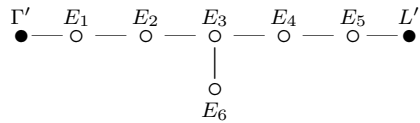
\newcommand{\linee}{\text{---}}

\[
\begin{array}{c@{\quad }c@\,c@\,c@\,c@\,c@\,c@\,c@\,c@\,c}
&\overset{\Gamma'}\bullet
\text{---}\overset{E_1}{\circ}&\linee&\overset{E_2}\circ
&\linee&\overset{E_3}\circ&\linee&\overset{E_4}{\circ}
&\linee&\overset{E_5}\circ\text{---}\overset{L'}\bullet
\\[-4pt]
&&&&&|
\\[-4pt]
&&&&&\underset{E_6}{\circ}
\end{array}
\]
Since $L$ and $\Gamma$ are both smooth and pass through $P$, we
have $Z\cdot \Gamma'=1=Z\cdot L'$, where $\Gamma'$ and $L'$ are
the proper transforms of $\Gamma$ and $L$ on $\tilde F$,
respectively. In the minimal resolution graph, $\Gamma'$ and $L'$
must be both attached at the end vertices $E_1$ or $E_5$. We claim
that they are not attached to the same vertex.

Suppose to the contrary that $\Gamma'\cdot E_1=1= L'\cdot E_1$. We
use the notion of a \textit{different} (see e.g.\ \cite{Sh2},
\cite{Pr4})
$$\Diff_C(0)=(K_F+C)|_C-K_C\,,$$ where $C$ is a curve on $F$
smooth at $P$. By adjunction,
$$(K_F+L)\cdot L=-2+\deg \Diff_L(0)\quad\text{and}\quad
(K_F+\Gamma)\cdot \Gamma=4+\deg \Diff_\Gamma(0)\,,$$ where
$\Diff_L(0)=\Diff_\Gamma(0)$ because of the local analytic
invariance of the different \cite{Pr4}. We have $(K_F+L)\cdot
L=-2/3$ and so $\deg \Diff_L(0)=4/3$. On the other hand,
$$(K_F+\Gamma)\cdot \Gamma=(-3L+7L)\cdot 7L=28L^2=28/3\,.$$
We deduce that $\deg \Diff_\Gamma(0)=16/3\neq \deg \Diff_L(0)$, a
contradiction. Thus we may assume that $\Gamma' \cdot E_1=1=L'
\cdot E_5$. Then by the symmetry of the resolution graph, it
follows that $\deg \Diff_L(0)=\deg \Diff_\Gamma (0)$, which is
again absurd by the same computation as above. \erem

\end{document}